\documentclass{amsart}
\usepackage{amsmath, amssymb}
\usepackage{amsthm}
\usepackage{commath}
\usepackage[mathscr]{eucal}
\usepackage{color}
\usepackage{amsfonts,amscd, enumerate}
\usepackage{tikz, graphicx, pgfplots}
\pgfplotsset{compat=1.3}
\newtheorem{thm}{Theorem}

\newtheorem{prop}[thm]{Proposition}
\renewenvironment{proof}{\par\noindent{\bf Proof.}}{$\square$\par\bigskip}
\newtheorem{lemma}[thm]{Lemma}
\newtheorem{remark}[thm]{Remark}
\newtheorem{cor}[thm]{Corollary}

\newtheorem{example}[thm]{Example}

\def\C{\mathbb C}

\def\log{\operatorname{log}}

\begin{document}
\title{Linear recurrent sequences, Markov chains and their applications in graph theory}
\author[Rebecca Carter]{Rebecca Carter}
\address{Department of Mathematics and Statistics, Queen's University, Kingston, Ontario, K7L 3N6, Canada.}
\email{   rebecca.carter@queensu.ca }

\author[M. Ram Murty]{M. Ram Murty}
\address{Department of Mathematics and Statistics, Queen's University, Kingston, Ontario, K7L 3N6, Canada.}
\email{murty@queensu.ca}

\thanks{Research of the second author was partially supported by an NSERC Discovery grant. \\ \phantom{abc}\today}
\begin{abstract}  
After a brief review of the key theorems 
concerning recurrent sequences, we give
an explicit computation of the inverse of the
Vandermonde matrix.  This will then be used
to derive a sub-exponential decay error terms in the ergodic theorem
of Markov chains.  Finally, we apply these
results to give estimates for the diameters
of directed graphs.

\end{abstract}
\maketitle
\section{Introduction}

The concept of recurrence is a fundamental theme
of science.  Indeed, knowledge can be defined as the recognition of
patterns and patterns are the manifestation of recurrence.  In mathematics, recurrent sequences
encode this phenomenon and enable us to predict
future movements of recurrent cycles.  In particular,
the theory of Markov chains that has now become
the foundation for many of the tools of the
internet age, ranging from the Google search engine
to data analytics, is essentially a theory of recurrent
sequences.  The limit theorem of
Markov chains (often referred to as the ergodic theorem of Markov chains) is fundamentally about 
the behavior of certain
recurrent sequences.  Surprisingly, very little
attention has been given to the study of error
terms in this celebrated Markov limit theorem.  
To the best of our knowledge, a ``sub-exponential rate of convergence'' is briefly alluded to in the text
by Koralov and Sinai 
\cite{KS} (see page 73 and Remark 5.10) and is indicated there as a ``remark''.
Moreover, their error term is not ``spectral''
in the sense that it depends on the entries of a
certain power of
the transition matrix rather than the eigenvalues
of the original transition matrix.
The purpose of this paper is to derive a
sub-exponential rate of convergence which is spectral
(see Theorem \ref{markov_convergence} below).
What is interesting about this bound is that
not only is it ``spectral'' but also that the error depends on the spacings between the eigenvalues, a phenomenon not noticed before.
In analytic number theory, knowledge of the spacings between
the zeros of the Riemann zeta function has led
to deeper knowledge about the distribution of prime numbers.  So, what we will see here is a similar
phenomenon.

After a brief review of the key theorems 
concerning recurrent sequences, we give
an explicit computation of the inverse of the
Vandermonde matrix.  This will then be used
to derive sub-exponential decay error terms in the ergodic theorem
of Markov chains.  Finally, we apply these
results to give new estimates for the diameters
of directed graphs.

We remark that the explicit determination of the Vandermonde matrix is not new (for example, it
appears as exercise 40 on page 36 of \cite{knuth}).
As far as we are aware, the first
formulation of the result in the form we
need it seems to be in the paper by 
Macon and Spitzbart \cite{macon-spitzbart}
written in 1958.  For the sake of completeness,
we provide a more direct proof of the result.
It is surprising that this result is rediscovered
again and again
in the literature.  The paper \cite{rawashdeh} published in 2019 is an example of this phenomenon.

As stated above, 
our primary goal here is to apply this result so as to give
an effective estimate for the error term in
the Markov limit theorem.  We begin with a 
brief review of linear recurrent sequences.

\section{Linear Recurrent Sequences}
A \textit{linear recurrent sequence} of order \(m\) is a sequence \(\{x_n\}_{n=0}^{\infty}\) that satisfies a recurrence relation of the form:
$$x_n=a_{m-1}x_{n-1}+\cdots+a_{0}x_{n-m}, $$
for \(n \ge m\). The first \(m\) terms \(\{x_0, \ldots, x_{m-1}\}\) are given as initial conditions. The \textit{characteristic polynomial} for the recurrence relation is the polynomial 
$$P(T):=T^m-a_{m-1}T^{m-1} -\cdots- a_1T - a_0. $$
Let $\lambda_1, \ldots, \lambda_m$ be the roots of the characteristic polynomial. Then, there exist polynomials $c_i(x) \in \mathbb{C}[x]$ such that
\begin{equation}\label{basic}
 x_n = \sum_{i=1}^m c_i(n) \lambda_i^n. 
 \end{equation}
This is called the closed form of the linear recurrent sequence. If a root $\lambda_i$ is simple, then the polynomial $c_i(x)$ is constant. 
The simplest method for deriving this is by the method of generating
functions.  One defines the formal power series
$$ \sum_{n=0}^\infty x_n T^n $$
and shows that this is a rational function
using the recurrence relation.  One then
uses a partial fraction decomposition to derive the explicit
formula for the $x_n$. There are many classical references
such as page 392 of Hardy \cite{hardy} or page
323 of \cite{graham} that provide a full explanation of how to use generating functions to find such solutions. In all these
derivations, the $c_i$'s occurring in the closed form expression
are often given in terms of the residues of the allied rational
function and never ``spectrally'' in terms of the $\lambda_i$.
Our goal in this paper is to provide a spectral formula for them
and then apply this to derive an explicit sub-exponential error term
for the classical limit theorem of Markov chains.
Other applications include computing diameters of directed graphs.

\subsection{Coefficients of a linear recurrent sequence}
Suppose that the roots $\lambda_i$ of a linear recurrent sequence $\{x_n\}_{n=0}^\infty$ are all simple. Then, the coefficients in its closed form (\ref{basic}) will be constants. We can then write the system of equations of the first $m$ terms of the sequence, in matrix notation: 
\[
\begin{bmatrix}
    x_1 \\
    x_2 \\
    \vdots \\
    x_m
\end{bmatrix} = 
\begin{bmatrix}
    \lambda_1 & \lambda_2 & \cdots & \lambda_m \\
    \lambda_1^2 & \lambda_2^2 & \cdots & \lambda_m^2 \\
    \vdots & \vdots & \ddots & \vdots \\
    \lambda_1^n & \lambda_2^n & \cdots & \lambda_m^n
\end{bmatrix} 
\begin{bmatrix}
    c_1 \\
    c_2 \\
    \vdots \\
    c_m
\end{bmatrix}.
\]
Here, the $m \times m$ matrix is a Vandermonde matrix with entries $\lambda_1, \ldots, \lambda_m$. In the case that the Vandermonde matrix is non-singular, it is possible to solve for the coefficients $c_i$ using the inverse of the matrix. Expressing these coefficients in terms of the roots of the sequence will have many applications, as we will see.

The determinant of this $m \times m$ Vandermonde matrix is
well-known to be
$$\det(V) = (-1)^{\binom{n}{2}}\prod_{i}\lambda_i\prod_{i>j}(\lambda_j-\lambda_i).$$
Therefore, if the roots $\lambda_1, \ldots, \lambda_m$ are non-zero and distinct, then $V$ is non-singular.
The explicit calculation of the inverse is not difficult
but surprisingly not so well-known.  As noted earlier, it appears as an
exercise on page 36 of \cite{knuth}.  But the
form that is given there is not useful for our application.
We therefore give an alternate derivation that is simple and 
helpful in our context.

Let $e_j(\lambda_1, \ldots, \lambda_n)$ be the $j^{th}$ degree \textit{elementary symmetric polynomial} in $n$ variables $\lambda_1, \ldots, \lambda_n$. Then, 
$$e_j(\lambda_1, \ldots, \lambda_n) = \sum\limits_{\substack{1 \leq k_1 < \cdots < k_{j} \leq n }} \lambda_{k_1}\cdots\lambda_{k_j}.$$
These polynomials will appear in the entries of the inverse of the Vandermonde matrix. It is useful to recall they can also be found as the coefficients of the monic polynomial with roots $-\lambda_i$.
More precisely, the elementary symmetric function, 
    \begin{equation*}
        e_{j}(\lambda_1,  \ldots, \lambda_n),
    \end{equation*}
    is the coefficient of \(x^{n-j}\) in the polynomial, \[\prod\limits_{\substack{1 \leq k \leq n} }(x + \lambda_k).\] 
The $j^{th}$ degree elementary symmetric polynomial in $n-1$ variables, 
$$\lambda_1, \ldots, \lambda_{i-1}, \lambda_{i+1}, \ldots, \lambda_{n}, $$ is conveniently denoted by $$e_{j}(\lambda, \ldots, \widehat{\lambda_i}, \ldots, \lambda_n),$$
where the ``hat'' on the $\lambda_j$ means that it is
absent in the evaluation.  With this notation in place, we can
state our results succinctly.

\begin{lemma} \label{Inverse of Vandermonde}
    The inverse of the \(n\times n\) Vandermonde matrix with distinct, non-zero entries \(\lambda_1, \ldots, \lambda_n\) is given by \(V^{-1} = [w_{ij}]\) where 
    \[w_{ij} = \frac{(-1)^{j-1}e_{n-j}(\lambda_1, \ldots, \widehat{\lambda_i}, \ldots, \lambda_n)}{\lambda_i \prod\limits_{\substack{1 \leq k \leq n \\ k \neq i} } (\lambda_i - \lambda_k)}.\]
\end{lemma}

\begin{proof}
    By definition of the inverse, \(V^{-1}V = I\), where \(I\) is the \(n \times n\) identity matrix. Therefore, 
    \begin{equation} \label{inverse_eq}
        \sum_{t=1}^n w_{it}\lambda_j^t = \delta_{ij},
    \end{equation}
    where $\delta_{ij}$ is the Kronecker symbol (equal to 1 if $i=j$
    and zero otherwise). The left-hand side of the equation is a degree $n$ polynomial in $\lambda_j$ with complex coefficients. Let
    \begin{equation} \label{inverse polynomial}
        G(x)= \sum_{t=1}^n w_{it}x^t.
    \end{equation}
    Then, the left-hand side of (\ref{inverse_eq}) can be written as $G(\lambda_j)$ and $G(x)$ vanishes 
    when $x$ belongs to the $n$-element set \(\{0, \lambda_1, \ldots, \lambda_{i-1}, \lambda_{i+1}, \ldots, \lambda_n\}\)
    because the right hand side vanishes when $j\neq i$.  Since $G(x)$ is of degree $n$, these must be its only roots. Therefore, we can factor the polynomial as, 
    \begin{equation} \label{inverse_eq2}
        G(x) = Ax\prod\limits_{\substack{1 \leq k \leq n \\ k \neq i} } (x - \lambda_k), 
    \end{equation}
    where $A \in \mathbb{C}$. Furthermore by (\ref{inverse_eq}), $G(\lambda_i)=\delta_{ii}= 1$. Hence, 
    \begin{equation}
        1 = A\lambda_i\prod\limits_{\substack{1 \leq k \leq n \\ k \neq i} } (\lambda_i - \lambda_k).
    \end{equation}
    We can then solve for $A$ and write, 
    \begin{equation}
        G(x)= \frac{x\prod\limits_{\substack{1 \leq k \leq n \\ k \neq i} } (x - \lambda_k)}{\lambda_i\prod\limits_{\substack{1 \leq k \leq n \\ k \neq i} } (\lambda_i - \lambda_k)}.
    \end{equation}
    Using the elementary symmetric polynomials, we can express the polynomial in the numerator in standard form,
    \begin{equation}
        G(x) = \frac{x
        \Big(\sum_{t=1}^n(-1)^{n-t} e_{n-t}(\lambda_1, \ldots, \widehat{\lambda_i}, \ldots, \lambda_n)x^{t-1}\Big)
        }
        {\lambda_i\prod\limits_{\substack{1 \leq k \leq n \\ k \neq i} } (\lambda_i - \lambda_k)}.
    \end{equation}
    Then, by matching the coefficients in this expression with those in (\ref{inverse polynomial}), we conclude that the $ij^{th}$ entry of $V^{-1}$ is
    \begin{equation}
        w_{ij} = \frac{
        (-1)^{n-j} e_{n-j}(\lambda_1, \ldots, \widehat{\lambda_i}, \ldots, \lambda_n)
        }
        {\lambda_i\prod\limits_{\substack{1 \leq k \leq n \\ k \neq i} } (\lambda_i - \lambda_k)}.
    \end{equation}
\end{proof}

We will now use the entries of the inverse of a Vandermonde matrix to solve for the complex coefficients of a linear recurrent sequence with non-zero, distinct roots. 

\begin{prop} \label{coefficients prop}
    Let \(\{x_n\}_{n=0}^\infty\) be a sequence such that \(x_n=\sum_{i=1}^m c_i \lambda_i^n\), where the roots \(\lambda_1, \ldots, \lambda_m\) are non-zero and distinct, and $c_i \in \C$. Then, 
    \begin{equation}
        c_i = \sum_{t=1}^m (-1)^{m-t}\frac{e_{m-t}(\lambda_1, \ldots, \widehat{\lambda_i}, \ldots, \lambda_m)}{\lambda_i \prod\limits_{\substack{1\leq k \leq m \\ k \neq i}}(\lambda_i-\lambda_k)}x_t.
    \end{equation}
\end{prop}

\begin{proof}
    Let \(V\) be the \(m \times m\) Vandermonde matrix with non-zero, distinct entries \(\lambda_1, \ldots, \lambda_m\). Then, 
    \begin{equation*}
        V \begin{bmatrix}
            c_1 \\
            \vdots \\
            c_m
        \end{bmatrix} = 
        \begin{bmatrix}
            x_1 \\
            \vdots \\
            x_m
        \end{bmatrix}, 
    \end{equation*}
    and $V$ has inverse $V^{-1}=[w_{ij}]$. Therefore, we can solve for the coefficients $c_i$,
    \begin{equation*}
        \begin{bmatrix}
            c_1 \\
            \vdots \\
            c_m
        \end{bmatrix} = 
        V^{-1} \begin{bmatrix}
            x_1 \\
            \vdots \\
            x_m
        \end{bmatrix}
        \qquad and \qquad 
        c_i = \sum_{t=1}^m w_{it}x_t.
    \end{equation*}
    Hence, by Lemma \ref{Inverse of Vandermonde}, 
    $$c_i = \sum_{t=1}^m (-1)^{m-t}\frac{e_{m-t}(\lambda_1, \ldots, \widehat{\lambda_i}, \ldots, \lambda_m)}{\lambda_i \prod\limits_{\substack{1\leq k \leq m \\ k \neq i}}(\lambda_i-\lambda_k)}x_t.$$
\end{proof}

One example of a linear recurrent sequence that is of particular interest to us arises from the powers of any square matrix over a commutative ring. 

The \textit{characteristic polynomial} of an $m\times m$ matrix $A$ is $p(\lambda) = \det(\lambda I - A).$ Its roots are the \textit{eigenvalues} of $A$. Let us recall:

\begin{thm}[Cayley-Hamilton]
    Every square matrix over a commutative ring satisfies its own characteristic equation. 
\end{thm}

Using the Cayley-Hamilton Theorem, we will prove that for any $m \times m$ matrix $A$, the entries of its powers $A^n$ form a linear recurrent sequence.   More precisely: 
\begin{thm} \label{recurrent_matrix_entries}
    Let $A=[a_{ij}]$ be an $m \times m$ matrix over $\mathbb{C}$ and $A^n=[a_{ij}^{(n)}]$. Then there exists $c_{ij}^{(r)} \in \mathbb{C}[x]$ for $1 \leq r \leq m$ such that, 
    $$a_{ij}^{(n)}= \sum_{r=1}^m c_{ij}^{(r)}(n) \lambda_r^n, $$
    where $\lambda_r$ are the eigenvalues of $A$. Furthermore, if the eigenvalues $\lambda_r$ are all simple, then $c_{ij}^{(r)}(x)$ are all constants.
\end{thm}

\begin{proof}
    Let 
    $$p(\lambda)=\lambda^m + a_{m-1}\lambda^{m-1} + \cdots + a_0,$$ 
    be the characteristic equation of $A$. By the Cayley-Hamilton Theorem, 
    $$ A^m = -a_{m-1}A^{m-1} - \cdots - a_0 I.$$
    Furthermore, for any $n > m$, 
    \begin{align*}
        A^n &= A^{n-m}(-a_{m-1}A^{m-1} - \cdots - a_0I) \\
        &= -a_{m-1}A^{n-1} - \cdots - a_0A^{n-m}.
    \end{align*}
    Therefore, for any $n > m$ and $i, j$ fixed, the $(i,j)-$entry of the $n^{th}$ power of $A$ is a linear combination of the $(i, j)$-entries of the preceding $m$ positive powers of the matrix. 
    In other words, $\{a_{ij}^{(n)}\}_{n >0}$ is a linear recurrent sequence of order $m$. Hence, there exists $c_{ij}^{(r)} \in \mathbb{C}[x]$ for $1 \leq r \leq m$ such that, 
    $$a_{ij}^{(n)}= \sum_{r=1}^m c_{ij}^{(r)}(n) \lambda_r^n,$$
    where $\lambda_1, \ldots, \lambda_m$ are the eigenvalues of $A$. Furthermore, $c_{ij}^{(r)}$ is a constant for every $1 \leq r \leq m$ if all of the eigenvalues are simple.
\end{proof}

This result 
has important implications for Markov theory.  It
implies that the $n$-step transition probabilities of a discrete Markov chain form a linear recurrent sequence. We will now briefly cover some basic background before considering the applications of this fact. 

\section{Applications to Markov Chains} \label{markov_chains_section}
Recall that
a sequence of random variables \((X_0, X_1, \ldots)\) is a discrete \textit{Markov chain} with finite \textit{state space} \(S\) and \textit{transition matrix} \(P\) if  
\begin{equation} \label{markov_property}
    Pr(X_{n+1} =s_j \mid X_0 = s_{i_0}, \ldots, X_n = s_{i_n}) = Pr(X_{n+1} = s_j \mid X_n = s_{i_n}),
\end{equation}  
for every \(n \geq 1\) and \(s_j, s_{i_0}, \ldots, s_{i_n} \in S\) such that  
\[
Pr(X_0 = s_{i_0}, \ldots, X_n = s_{i_n}) > 0.
\]  
Here, \(S\) represents the possible states of the system, \(X_n\) indicates the state of the system at time \(n\), and the sequence \(X_0, X_1, \ldots\) records the history of the system. The \textit{Markov property} (\ref{markov_property}) implies that the probability of transitioning from state \(s_i\) to state \(s_j\) depends only on the current state \(s_i\) and not on the preceding history.

The \(ij^{th}\) entry of the transition matrix \(P\) is given by  
\[
p_{ij} = Pr(X_{n+1} = s_j \mid X_n = s_i).
\]  
These entries are referred to as the \textit{transition probabilities} and they do not depend on $n$. The entries of \(P^n = [p_{ij}^{(n)}]\) represent the \(n\)-step transition probabilities, where \(p_{ij}^{(n)}\) is the probability that the system transitions from state \(s_i\) to state \(s_j\) in \(n\) time-steps. 

While a Markov chain can have a countably infinite state space and evolve in continuous time, we will assume when referring to a Markov chain that it is discrete and finite (unless otherwise stated). The reader unfamiliar with this
classical topic can find the relevant introduction in many
places such as Chapter 5 of \cite{KS}

A Markov chain can also be viewed as a directed graph, where its transition matrix corresponds to the graph’s weighted adjacency matrix. The set of states is represented by a set of vertices and a non-zero transition probability \(p_{ij}\) corresponds to a directed edge \((v_i, v_j)\), which is weighted by \(p_{ij}\). 

A Markov chain is considered \textit{irreducible} if, for any \(s_i, s_j \in S\), there exists \(n \geq 1\) such that \(p_{ij}^{(n)} > 0\). The corresponding directed graph is strongly connected in this case. The \textit{period} of a state \(s_i\) is defined as \(\text{gcd}(\mathcal{T}(s_i))\), where \(\mathcal{T}(s_i) = \{n \geq 1 \mid p_{ii}^{(n)} > 0\}\). If a Markov chain is irreducible, then all states share the same period, which is referred to as the period of the Markov chain. If the period is \(1\), the Markov chain is said to be \textit{aperiodic}.  

An \textit{initial distribution} is a row vector \(\boldsymbol{\pi} = [\pi_1, \ldots, \pi_m]\) such that  
\[
\pi_i = Pr(X_0 = s_i),
\]  
with \(\sum_i \pi_i = 1\). 
Since the row sum of $P$ is 1, it has one of its eigenvalues
equal to 1.  
If an initial distribution \(\boldsymbol{\pi}\) satisfies \(\boldsymbol{\pi}P = \boldsymbol{\pi}\), it is called a \textit{stationary distribution}. 
In other words, $\boldsymbol{\pi}$ is a left eigenvector
corresponding to the eigenvalue 1.  
The following theorem guarantees the existence of a unique stationary distribution for an aperiodic and irreducible Markov chain. 

\begin{thm}[Fundamental Theorem of Markov Chains] \label{fund_thm_markov}
    If a Markov chain is irreducible and aperiodic, then it has a unique stationary distribution \(\boldsymbol{\pi}\). Furthermore,  
    \[
    \lim_{n \to \infty} p_{ij}^{(n)} = \pi_j.
    \]
\end{thm}

\subsection{Convergence Theorem}

A natural question is at what rate does an irreducible, aperiodic Markov chain converge to its unique stationary distribution? It is a well-established result that this is an exponential process \cite{billingsley, KS, Levin}. 
Known estimates for the rate of convergence are not ``spectral.''  
We will present a new proof of the convergence theorem which will derive explicit error estimates in terms of the spectrum of the transition matrix. These spectral error estimates will provide meaningful insights into the mixing time of a Markov chain as well as have applications in other fields such as graph theory. 

Our approach will make use of the well-known Perron-Frobenius Theorem. 

\begin{thm}[Perron-Frobenius] \label{PF_theorem}
    Suppose that $P$ is a non-negative matrix which is irreducible. Then $P$ has a real, positive eigenvalue $\lambda_1$ such that:
    \begin{enumerate}
        \item[(a)] Any other eigenvalue $\alpha$ satisfies $\abs{\alpha} \leq \lambda_1$.
        \item[(b)] $\lambda_1$ has multiplicity 1.
        \item[(c)] $\lambda_1 \leq \max_j (\sum_k a_{jk})$ and $\lambda_1 \leq \max_k (\sum_j a_{jk})$.
        \item[(d)] If $P$ has exactly \textit{t} eigenvalues of maximum modulus $\lambda_1$, then there is a non-singular matrix $C$ such that,
        $$C^{-1}PC=\begin{bmatrix}
            \Lambda & 0 \\
            0 & J
        \end{bmatrix}$$
        where $\Lambda$ is a $t \times t$ diagonal matrix $\text{diag}(\lambda_1, \lambda_1\zeta_t, \ldots, \lambda_1 \zeta_t^{t-1})$ with $\zeta_t = e^{2\pi i/t}$ and $J$ is a Jordan matrix whose diagonal elements are all strictly less than $\lambda_1$ in modulus. 
        \item[(e)] If $P$ is aperiodic, then $t=1$.  
    \end{enumerate}
\end{thm}

\begin{remark}
    For non-negative, irreducible matrices, we will refer to this eigenvalue $\lambda_1$ as the dominant eigenvalue. Furthermore, we will let 
    $$\rho := \max_{\abs{\lambda_i} \neq \lambda_1}\abs{\lambda_i}.$$
\end{remark}

As stated in the previous section, Theorem \ref{recurrent_matrix_entries} shows that the $n$-step transition probabilities of a Markov chain form a linear recurrent sequence. Our proof of the convergence theorem of a Markov chain will consider this recurrent process and apply Proposition \ref{coefficients prop}.
    
\begin{thm}[Convergence Theorem]\label{markov_convergence}
    Let $P$ be the transition matrix of an irreducible, aperiodic Markov chain and $\boldsymbol{\pi}$ be its unique stationary distribution. Suppose the eigenvalues of $P$, $\lambda_1, \ldots, \lambda_m$, are all non-zero and simple. Then, 
    $$\abs{p_{ij}^{(n)} - \pi_j} \leq \Phi \rho^n, $$
    where $$\Phi = \sum\limits_{k=2}^m 
     \frac{\abs{\sum_{l=1}^m(-1)^{m-l} e_{m-l}(\lambda_1, \ldots, \widehat{\lambda_k}, \ldots, \lambda_m)p_{ij}^{(l)}}}{\abs{\lambda_k} \prod\limits_{\substack{1\leq t \leq m \\ t \neq k}}\abs{\lambda_k-\lambda_t}}.$$
\end{thm}

\begin{remark}
Since $|\rho|<1$, 
    this theorem provides an exponential decay rate of convergence for the Markov process which is dictated by the minimum gap between eigenvalues and their maximum absolute value less than the dominant eigenvalue.  A similar (non-spectral)
    bound was obtained in the second author's joint paper \cite{MP}.
\end{remark}
 
\begin{proof}
    Let $P$ be the transition matrix of an irreducible, aperiodic Markov chain. Then, by the Fundamental Theorem of Markov chains, there exists a unique stationary distribution $\boldsymbol{\pi}$. Furthermore, by the Perron-Frobenius Theorem, the dominant eigenvalue of $P$ is $\lambda_1 = 1$ and $\abs{\lambda_2} = \rho$. Suppose the eigenvalues of $P$ are simple. Then, by Theorem \ref{recurrent_matrix_entries}, there exists $c_{ij}^{(k)} \in \mathbb{C}$ such that, 
    $$p_{ij}^{(n)} = \sum_{k=1}^m c_{ij}^{(k)} \lambda_k^n.$$
    By the Fundamental Theorem of Markov chains, 
    $\pi_j = \lim_{n \to \infty} p_{ij}^{(n)}.$
    Hence, $\pi_j = c_{ij}^{(1)}.$ Therefore, 
    $$\abs{p_{ij}^{(n)} - \pi_j} \leq \sum_{k=2}^m \abs{c_{ij}^{(k)}}\rho^n.$$ By Proposition \ref{coefficients prop}, 
    $$\abs{c_{ij}^{(k)}} = \frac{\abs{\sum_{l=1}^m(-1)^{m-l} e_{m-l}(\lambda_1, \ldots, \hat{\lambda_k}, \ldots, \lambda_m)p_{ij}^{(l)}}}{\abs{\lambda_k} \prod\limits_{\substack{1\leq t \leq m \\ t \neq k}}\abs{\lambda_k-\lambda_t}},$$
    which proves the theorem. 
\end{proof}

The proof can be easily modified to account for the case where the eigenvalues are not simple or are equal to zero. In this case, we require only that the Markov matrix is diagonalizable. By considering the diagonalization of the matrix, we can express the transition probabilities in the closed form of a linear recurrent sequence with constant coefficients. We can group together the terms of this summation that share the same eigenvalues and eliminate the terms with $\lambda_i=0$. The summation will then be indexed from $1$ to $M$ where $M$ is the number of non-zero and distinct eigenvalues. 

It can be useful to further simplify Theorem \ref{markov_convergence} to no longer depend on the first $m$ $n$-step transition probabilities. 

\begin{cor}
    Let $P$ be the transition matrix of an irreducible, aperiodic Markov chain and $\boldsymbol{\pi}$ be its unique stationary distribution. Suppose the eigenvalues of $P$, $\lambda_1, \ldots, \lambda_m$, are all non-zero and simple. Then,
    $$ \abs{p_{ij}^{(n)}-\pi_j} \leq \Psi\rho^{n-1},$$
    where $\Psi=\sum\limits_{k=2}^m \Bigg( \prod\limits_{\substack{t=1 \\ t \neq k}}^m \frac{1 + \abs{\lambda_t}}{\abs{\lambda_k -\lambda_t}} \Bigg)$.
\end{cor}

\begin{proof}
    We can amend the previous proof to further bound the complex coefficients. Recall, 
    $$p_{ij}^{(n)} - \pi_j = \sum\limits_{k=2}^m\Bigg(
    \sum_{l=1}^m(-1)^{m-l}\frac{e_{m-l}(\lambda_1, \ldots, \widehat{\lambda_k}, \ldots, \lambda_m)}{\lambda_k \prod\limits_{\substack{1\leq t \leq m\\ t \neq k}}(\lambda_k-\lambda_t)}p_{ij}^{(l)} \Bigg) \lambda_k^n. $$
    Then, 
    $$\abs{p_{ij}^{(n)}-\pi_j} \leq \sum\limits_{k=2}^m\Bigg( \frac{\rho^{n-1}}{\prod\limits_{\substack{1\leq t \leq m\\ t \neq k}}\abs{\lambda_k-\lambda_t}}\sum\limits_{l=1}^me_{m-l}(\abs{\lambda_1}, \ldots, \widehat{\abs{\lambda_k}}, \ldots, \abs{\lambda_m})\Bigg).$$
    Recall that $e_{m-l}(\abs{\lambda_1}, \ldots, \widehat{\abs{\lambda_k}}, \ldots, \abs{\lambda_m})$ is the coefficient of $x^{l-1}$ in the polynomial $$\prod\limits_{\substack{l=1\\ l \neq k}}^m(x+\abs{\lambda_l}).$$
    Therefore, by setting $x=1$, 
    $$\sum\limits_{l=1}^me_{m-l}(\abs{\lambda_1}, \ldots, \widehat{\abs{\lambda_k}}, \ldots, \abs{\lambda_m})=\prod\limits_{\substack{l=1\\ l \neq k}}^m(1+\abs{\lambda_l}).$$
    This proves,  
    $$\abs{p_{ij}^{(n)}-\pi_j} \leq \Psi\rho^{n-1}, $$
    where $\Psi=\sum\limits_{k=2}^m\Bigg( \prod\limits_{\substack{t=1 \\ t \neq k}}^m\frac{1 + \abs{\lambda_t}}{\abs{\lambda_k -\lambda_t}} \Bigg)$.
\end{proof}

The \textit{mixing time} of a Markov chain measures the amount of time required for the distance to the stationary distribution to be small. It has been shown that that the smaller $\rho$ is, the faster the Markov chain will mix \cite{Levin}. The explicit error estimates presented in Theorem \ref{markov_convergence} show that mixing time is also dependent upon the gaps between the eigenvalues of the transition matrix $P$. If these gaps are increased, resulting in the eigenvalues exhibiting a sort of ``repellent" behaviour, then the mixing time will decrease. 

As previously stated, these explicit error estimates in terms of the spectral data of the Markov chain, have applications beyond just the rate of convergence of the Markov chain. When considering the Markov chain as a weighted directed graph, the error estimates can also derive a spectral upper bound on the diameter of the graph. 

\section{Applications to Graph Theory}
A \textit{graph} $\Gamma = (V, E)$ consists of a set of \textit{vertices} $V$ and a set of \textit{edges} $E$. Each edge $(v_i, v_j) \in E$ connects two endpoints $v_i, v_j \in V$. In a directed graph, the pair of vertices representing an edge is ordered, where $v_i$ is designated as the tail and $v_j$ as the head. The \textit{indegree} of a vertex $v \in V$ of a directed graph $\Gamma = (V, E)$ is defined as the number of edges in $E$ for which $v$ is the head, while the \textit{outdegree} of $v$ is the number of edges for which $v$ is the tail. For an undirected graph, these quantities are identical, and we simply refer to the \textit{degree} of the vertex.

A \textit{walk} in a directed graph $\Gamma = (V, E)$ is represented as an alternating sequence of vertices and edges,
$$x_0, e_1, x_1, \ldots, e_n, x_n,$$
where the tail and head of each edge $e_i$ are $x_{i-1}$ and $x_i$, respectively, for $1 \leq i \leq n$. The \textit{length} of the walk is defined as $n$, the number of edges it contains. A \textit{path} is a walk such that no vertex $v_i$ appears more than once. 

The \textit{distance} from a vertex $v_i$ to another vertex $v_j$ is the length of the shortest path starting at $v_i$ and ending at $v_j$, it is denoted by $d(v_i, v_j)$. The maximum of these distances is the \textit{diameter} of the graph, given by
$$D(\Gamma) = \max_{v_i, v_j \in V} d(v_i, v_j).$$
For undirected graphs, the edges are not directional, so $d(v_i, v_j) = d(v_j, v_i)$. If there does not exist a path of any length from $v_i$ to $v_j$ then $d(v_i, v_j) = \infty$ and hence $D(\Gamma) = \infty$. A directed graph is \textit{strongly connected} if it has finite diameter. For undirected graphs, we would say it is \textit{connected}. 

\subsection{Diameters of directed graphs}
The diameter of a directed (or undirected) graph is a measure of its size and connectivity. For the remainder of this section, any graph will be assumed to have finite diameter unless otherwise stated. Previous work on upper bounds for graph diameters includes \cite{Chung}, where Chung derives a spectral upper bound on the diameter of \textit{k}-regular graphs. This is interesting in the context of Ramanujan
graphs.  In a later paper, she \cite{chung2} derived similar
results for directed graphs in terms of the
eigenvalues of the Laplacian of the graph,
whereas our results are in terms of the
eigenvalues of the Markov matrix (which is a weighted
adjacency matrix).  In general, there is no relation
between the spectrum of the Laplacian and the
spectrum of the adjacency matrix.
Thus, one may take the minimum of the Chung bound
and our bound to get an optimal estimate for the diameter.

In this section, we will show that the spectral error estimates derived in Theorem \ref{markov_convergence} can be used to derive an upper bound on the diameter of a directed graph which agrees with Chung's result for the undirected $k$-regular case. 

Recall that a Markov chain is a weighted directed graph, where the transition matrix $P$ is the weighted adjacency matrix. Now we will show that is it is possible to construct a Markov chain from any directed graph $\Gamma$. While there are many possible constructions for the Markov chain, for any of them if the weights are removed, you are left with $\Gamma$. Hence, as a weighted directed graph, the Markov chain has the same diameter as $\Gamma$. 

Starting with a strongly connected directed graph $\Gamma$, choose values $p_{ij}$ that satisfy the following requirements: For every $i, j \in \{1, \ldots, m\}$,  $0 \leq p_{ij} \leq 1$,  $p_{ij} = 0$ if and only if $(v_i, v_j) \notin E(\Gamma)$ and $\sum_{j=1}^m p_{ij} = 1$.  Then, the matrix $P$ with entries $p_{ij}$ is the transition matrix of a Markov chain. We will refer to this matrix as a \textit{Markov matrix for $\Gamma$}. While there does not necessarily exist a unique Markov matrix for any given directed graph $\Gamma$, the  construction in example \ref{markov_construction_example} works for any graph. 

\begin{example} \label{markov_construction_example}
    Let $\Gamma$ be a directed graph with $m$ vertices and $d^+(v_i)$ be the outdegree of the vertex $v_i \in V(\Gamma)$. Let, 
    \begin{equation*}
    p_{ij}= \begin{cases}
        \frac{1}{d^+(v_i)} &  (v_i, v_j) \in E \\
        0 & (v_i, v_j) \not\in E .
        \end{cases}
    \end{equation*}
    Then, $P=[p_{ij}]$ is a Markov matrix for $\Gamma$. 
\end{example}

For any Markov matrix $P$ for a directed graph $\Gamma$, there exists a walk of length $n$ from $v_i$ to $v_j$ if and only if $p_{ij}^{(n)}>0$. Since a path is a walk with no repeated $vertex$, 
$$d(v_i, v_j) = \min\{n \mid p_{ij}^{(n)} > 0\}$$
and the diameter \( D(\Gamma) \) is given by
$$D(\Gamma) = \max\limits_{{v_i, v_j}}d(v_i, v_j) = \max\limits_{i, j}\big(\min\{n \mid p_{ij}^{(n)} > 0\}\big).$$

The subexponential convergence of a Markov chain to its unique stationary distribution provides an upper bound on the diameter of the underlying directed graph. 

\begin{thm}\label{directed_diameter}
    Let $P$ be a Markov matrix for a strongly connected graph $\Gamma$. If $P$ is aperiodic and its eigenvalues, $\lambda_1, \ldots, \lambda_m$ are all non-zero and simple, then, 
    $$ D(\Gamma) \leq \max_j \left\lceil \frac{\log\bigg(\Phi / \pi_j \bigg)}
    {\log(1/\rho)} \right\rceil,$$
    where $$\Phi = \sum\limits_{k=2}^m 
     \frac{\abs{\sum_{l=1}^m(-1)^{m-l} e_{m-l}(\lambda_1, \ldots, \widehat{\lambda_k}, \ldots, \lambda_m)p_{ij}^{(l)}}}{\abs{\lambda_k} \prod\limits_{\substack{1\leq t \leq m \\ t \neq k}}\abs{\lambda_k-\lambda_t}} \in \mathbb{R}_{\geq 0}.$$ 
\end{thm}

\begin{proof}
    Let $\Gamma$ be a strongly connected directed graph with Markov matrix $P$. Suppose $P$ is aperiodic and its eigenvalues are all non-zero and simple. Then by Theorem \ref{markov_convergence},
    $$\abs{p_{ij}^{(n)}-\pi_j} \leq \rho^{n}\Phi.$$
    Hence, $p_{ij}^{(n)} > 0$ if, 
    $$ n > \frac{\log\bigg(\Phi / \pi_j \bigg)}
    {\log(1/\rho)}.$$
    Recall that the transition probability $p_{ij}^{(n)}$ is strictly positive if and only if there exists a path of length $n$ from $v_i$ to $v_j$ in the graph $\Gamma$. Therefore, we can conclude, 
    $$ D(\Gamma) \leq \max_j \left\lceil \frac{\log\bigg(\Phi / \pi_j \bigg)}
    {\log(1/\rho)} \right\rceil .$$
\end{proof}

We can show that this upper bound agrees with Theorem 1 of \cite{Chung} by treating undirected graphs as a special case of directed graphs where the existence of an undirected edge $(v_i, v_j) \in V(\Gamma)$ requires both $p_{ij}$ and $p_{ji}$ to be nonzero.  

\begin{cor}[F.R.K. Chung] \label{regular_diameter}
    Let $\Gamma$ be an undirected \textit{k}-regular graph with adjacency matrix $A$. Let $P=\frac{1}{k}A$, then $P$ is a Markov matrix for $\Gamma$. Suppose $P$ is aperiodic with non-zero simple eigenvalues. Then, 
    $$D(\Gamma) \leq \left\lceil \frac{\log(m-1)}{\log(k/ \tau)} \right\rceil,$$
    where $\tau$ is the second largest eigenvalue of $A$ in terms of absolute value. 
\end{cor}

\begin{proof}
    Let $\Gamma$ be an undirected graph, then its adjacency matrix $A$ is symmetric and hence so is the associated Markov matrix $P=\frac{1}{k}A$. Let $\lambda_1, \ldots, \lambda_m$ be the eigenvalues of $P$ listed in descending order according to absolute value. Suppose they are all simple and non-zero, then $\rho=\abs{\lambda_2}$.  

    By Theorem \ref{directed_diameter}, 
    $$ D(\Gamma) \leq \max_j \left\lceil \frac{\log\bigg(\Phi / \pi_j \bigg)}
    {\log(1/\rho)} \right\rceil,$$
    where $p_{ij}^{(n)}=\sum_{k=1}^m c_{ij}^{(k)} \lambda_k^n$ and $\Phi = \sum\limits_{k=2}^m \abs{c_{ij}^{(k)}}$. We will show that $\Phi \leq \frac{m-1}{m}$ and $ \pi_j = \frac{1}{m}$. 

    As a symmetric matrix, $P$ can be diagonalized by an orthogonal matrix, $B=[b_{ij}]$, whose columns are orthonormal eigenvectors of $P$. Therefore, 
    $$p_{ij}^{(n)} = \sum_{k=1}^m b_{ik} \lambda_k^n b_{jk},$$
    and so $c_{ij}^{(k)}=b_{ik}b_{jk}$. By the Perron-Frobenius theorem, the dominant eigenvalue of $P$ is $\lambda_1 = 1$ and $\abs{\lambda_k} < 1$ for $k \ge 2$. The dominant eigenvalue has right eigenvector of all 1's. Therefore, the first column of $B$ is $[1/\sqrt{m}, \ldots, 1/\sqrt{m}]^T$. Applying the Cauchy-Schwarz inequality, we achieve the following upper bound, 
    \begin{align*}
        \Phi &\leq  \sqrt{\sum\limits_{k=2}^m \abs{b_{ik}}^2}\sqrt{\sum\limits_{k=2}^m \abs{b_{jk}}^2} \\
        &= \sqrt{1-\abs{b_{i1}}^2}\sqrt{1-\abs{b_{j1}}^2} \\
        &= \frac{m-1}{m}.
    \end{align*}
    Furthermore, by the Fundamental Theorem of Markov Chains, 
    \begin{align*}
        \pi_j &= \lim_{n \to \infty} p_{ij}^{(n)} \\
        &= b_{i1}b_{j1} \\
        &=\frac{1}{m}.
    \end{align*}
    Since $P=\frac{1}{k}A$, the eigenvalues of $A$ are $\mu_i = k\lambda_i$. Therefore, $\rho = \tau/k$ where $\tau$ is the second largest eigenvalue of $A$ in absolute value. Hence, 
    $$\max_j \left\lceil \frac{\log\bigg(\Phi / \pi_j \bigg)}
    {\log(1/\rho)} \right\rceil \leq \left\lceil \frac{\log(m-1)}
    {\log(k/\tau)} \right\rceil, $$
    which proves the Corollary.
\end{proof}

Corollary \ref{regular_diameter} shows how we can treat undirected graphs as a special case of directed graphs when constructing Markov matrices for them. Furthermore, in the proof of the Corollary, we make use of the diagonalization of a symmetric Markov matrix of an undirected $k$-regular graph. It is useful to note here that if we allow ourselves to add self-loops to each vertex, then any undirected graph $\Gamma$ has a symmetric Markov matrix. All edges will be weighted with $1/d$ where $d$ is the maximum outdegree of any vertex in the graph and for each vertex $v_i$, its self-loop will be weighted with $1-d(v_i)/d$. The addition of these self-loops does not affect the diameter of the graph, so in the case that this Markov chain is aperiodic, irreducible and its eigenvalues are simple and non-zero, we can show that $D(\Gamma) \leq  \left\lceil \frac{\log(m-1)}{\log(1/\rho)} \right\rceil$.

\section{Concluding remarks}

There is a deep and foundational relationship between the theory of Markov chains and graph theory. Exploring their intersection offers insights with relevance to the emerging field of network science. Many of the results we obtained are well positioned to find meaningful applications in this context. Indeed, some of this work arose in the context of studying neural networks
\cite{MP} with a view to identifying ``brain hubs''
or centers of ``high'' neural activity.  Undoubtedly,
the symbiosis between the theory of Markov chains, graph theory and other branches
of science will expand in all directions.
$$\quad $$
\noindent{Acknowledgements.}  We thank Professors
Francesco Cellarosi and Brad Rodgers for their comments on an earlier version of this paper.

\end{document}